\documentclass[reqno,final]{amsart}
\usepackage{amsmath,amssymb,amsthm,amscd,verbatim}
\usepackage{color}
\usepackage{lscape}
\usepackage{fullpage}
\usepackage{stmaryrd}
\usepackage[notref,notcite,color]{showkeys}
\usepackage{epic,eepic}
\usepackage[dvips]{graphicx}
\usepackage[T1]{fontenc}

\numberwithin{equation}{section}
\usepackage{textcomp}
\usepackage{amsbsy}

\newtheorem{theorem}{Theorem}[section]

\newtheorem{lemma}[theorem]{Lemma}

\newtheorem{corollary}[theorem]{Corollary}

\theoremstyle{remark}
\newtheorem{remark}[theorem]{Remark}

\newcommand{\Image}{\mathrm{Im}}

\newcommand{\Ent}{{\operatorname{Ent}}}

\newcommand{\supp}{{\operatorname{supp}}}
\def\Ex {{\mathbb E}}
\def\Var {{\mathrm{Var}}}

\def\QQ{\mathcal{Q}}
\def\UU{\mathcal{U}}
\def\LL{\mathcal{L}}
\def\bsf{{\sf b}}

\newcommand{\ignore}[1]{}

\newcommand{\NN}{\mathbb{N}}

\def\II{\mathcal{I}}

\begin{document}

\title{The entropy of random-free graphons and properties}

\author{Hamed Hatami}
\address{School of Computer Science, McGill University, Montreal, Canada.}
\email{hatami@cs.mcgill.ca}
\author{Serguei Norine}
\address{Department of Mathematics \& Statistics, McGill University, Montreal,
Canada.}
\email{snorin@math.mcgill.ca}
\thanks{}

\begin{abstract}
Every graphon defines a random graph on any given number $n$ of vertices. It was known that the graphon is random-free if and only if the entropy of this random graph is subquadratic. We prove that for random-free graphons, this entropy can grow as fast as any subquadratic function.
However, if the graphon belongs to the closure of a random-free graph property, then the entropy is $O(n \log n)$. We also give a simple
construction of a non-stepfunction random-free graphon for which this entropy is linear, refuting a conjecture of Janson.

\end{abstract}

\maketitle

\section{Introduction \label{sec:intro}}
In recent years a theory of convergent sequences of dense graphs has been
developed.  One can
construct a limit object for  such a sequence in the form of certain symmetric
measurable functions called
graphons. Every graphon defines a random graph on any given number of vertices. In~\cite{HJS}
several facts about the asymptotics of the entropies of these random variables are established. These results provide a good understanding of the situation when the graphon is not ``random-free''. However in the case of the random-free graphons they completely trivialize. The purpose of this article is to study these entropies in the case of the random-free graphons.

\subsection{Preliminaries}
For every natural number $n$, denote $[n]:=\{1,\ldots, n\}$. In this paper all
graphs are simple and finite. For a graph $G$, let $V(G)$ and $E(G)$,
respectively denote the set of the vertices and the edges of $G$.
Let $\mathcal{U}$ denote set of all  graphs up to an
isomorphism. Moreover, for $n \ge 0$, let $\mathcal{U}_n \subset \mathcal{U}$
denote the set of all graphs in $\mathcal{U}$ with exactly $n$ vertices.
We will usually work with labeled graphs.  For every $n \ge 1$, denote by
$\LL_n$ the set of all graphs with vertex set
$[n]$.

The \emph{homomorphism density} of a graph $H$ in a graph $G$, denoted by
$t(H;G)$, is the probability that a random mapping $\phi:V(H) \to V(G)$ preserves adjacencies,
i.e. $uv \in E(H)$ implies $\phi(u)\phi(v) \in E(G)$.
The \emph{induced density} of a graph $H$ in a graph $G$, denoted by $p(H;G)$,
is the probability that a random \emph{embedding} of the vertices of $H$ in the
vertices of $G$ is an embedding of $H$ in $G$.

We call a sequence of finite
graphs $(G_n)_{n=1}^\infty$ \emph{convergent} if for every finite graph $H$,
the sequence $\{p(H;G_n)\}_{n=1}^\infty$ converges. It is not difficult to
construct convergent sequences $(G_n)_{n=1}^\infty$ such that their limits
cannot be recognized as graphs, i.e. there is no graph $G$, with $\lim_{n
\rightarrow \infty} p(H;G_n) = p(H;G)$ for every $H$. Thus naturally one
considers $\overline{\mathcal{U}}$,
the completion of  $\mathcal{U}$ under this notion of convergence.
It is not hard to see that $\overline{\mathcal{U}}$ is a compact
metrizable space which contains $\mathcal{U}$ as a dense subset. The elements of the
complement  $\mathcal{U}^\infty := \overline{\mathcal{U}} \setminus \mathcal{U}$
are called \emph{graph limits}. Note that a sequence of graphs $(G_n)_{n=1}^\infty$ converges
to a graph limit $\Gamma$ if and only if $|V(G_n)| \to \infty$ and $p(H;G_n) \to
p(H;\Gamma)$ for every graph $H$. Moreover, a graph limit is uniquely determined
by the numbers $p(H;\Gamma)$ for all $H \in \mathcal{U}$.

It is shown in~\cite{MR2274085} that every graph limit $\Gamma$ can be
represented by a \emph{graphon}, which is a symmetric measurable function
$W:[0,1]^2 \to [0,1]$. The set of all graphons are denoted by $\mathcal{W}_0$.
Given a graph $G$ with vertex set $[n]$, we
define the corresponding graphon $W_G: [0,1]^2 \rightarrow \{0,1\}$ as follows.
Let $W_G(x,y) :=
A_G(\lceil x n \rceil,\lceil y n \rceil)$ if $x,y \in (0,1]$, and if
$x=0$ or $y=0$, set $W_G$ to $0$. It is easy to see that if $(G_n)_{n=1}^\infty$ is a graph
sequence that converges to a graph limit $\Gamma$, then for every graph $H$,
$$ p(H;\Gamma)= \lim_{n \to \infty} \Ex\left[\prod_{uv \in E(H)}
W_{G_n}(x_u,x_v) \prod_{uv \in E(H)^c} (1-W_{G_n}(x_u,x_v))\right],$$
where $\{x_u\}_{u \in V(H)}$ are independent random variables taking values in
$[0,1]$ uniformly, and $E(H)^c=\{uv: u \neq v, uv \not\in E(H)\}$. Lov\'asz and
Szegedy~\cite{MR2274085}
showed that for every graph limit $\Gamma$, there exists a graphon $W$ such that
for every graph $H$, we have $p(H;\Gamma)= p(H;W)$ where
$$p(H;W):= \Ex\left[\prod_{uv \in E(H)} W(x_u,x_v) \prod_{uv
\in E(H)^c} (1-W(x_u,x_v))\right].$$
Furthermore, this graphon is unique in the following sense: If $W_1$ and $W_2$
are two different graphons representing the same graph limit, then there exists
a measure-preserving map $\sigma:[0,1] \to [0,1]$ such that
\begin{equation}
 \label{eq:graphon_uniquenss}
W_1(x,y) = W_2(\sigma(x),\sigma(y)),
\end{equation}
almost everywhere~\cite{MR2594615}. With these considerations, sometimes we shall not
distinguish between the graph limits and their corresponding graphons. We define the $\delta_1$ distance of two graphons
$W_1$ and $W_2$ as
$$\delta_1(W_1,W_2) = \inf \|W_1-W_2 \circ \sigma\|_1,$$
where the infimum is over all measure-preserving maps $\sigma:[0,1] \to [0,1]$.

A graphon $W$ is called a \emph{stepfunction}, if there is a partition of $[0,1]$ into a finite number of
measurable sets $S_1,\ldots,S_n$ so that $W$ is constant on every $S_i \times S_j$. The partition classes will be
called the \emph{steps} of $W$.

Let $W$ be a graphon and $x_1,\ldots,x_n \in [0,1]$. The random graph $G(x_1,\ldots,x_n,W) \in \LL_{n}$
is obtained by including the edge $ij$ with probability $W(x_i,x_j)$, independently for all pairs $(i,j)$
with $1 \le i < j \le n$. By picking $x_1,\ldots,x_n$ independently and uniformly at random from $[0,1]$, we obtain the random graph
$G(n,W) \in \LL_{n}$. Note that that for every $H \in \LL_n$,
$$\Pr[G(n,W) = H] = p(H;W).$$

\subsection{Graph properties and Entropy\label{sec:mainI}}

A subset of the set $\mathcal{U}$ is called a \emph{graph class}. Similarly a
\emph{graph property} is a property of graphs that is invariant under graph
isomorphisms. There is an obvious one-to-one correspondence between graph
classes and graph properties and we will not distinguish between a graph
property and the corresponding class. Let $\QQ \subseteq \mathcal{U}$ be
a graph class. For every $n>1$, we denote by $\QQ_n$ the set of graphs in $\QQ$
with exactly $n$ vertices. We let $\overline{\QQ} \subseteq \overline{\UU}$ be
the closure of $\QQ$ in $\overline{\UU}$.

Define the \emph{binary entropy} function $h:[0,1] \mapsto \mathbb{R}_+$ as
$h(x)=-x \log(x) - (1-x) \log(1-x)$ for $x \in (0,1)$ and $h(0)=h(1)=0$ so that
$h$ is continuous on $[0,1]$ where here and throughout the paper $\log(\cdot)$
denotes the logarithm to the base $2$. The \emph{entropy} of a graphon $W$ is defined as
$$\Ent(W) := \int_0^1 \int_0^1  h(W(x,y)) dx dy.$$
Note that it follows from the uniqueness result (\ref{eq:graphon_uniquenss})
that  entropy is a function of the underling graph limit, and
it does not depend on the choice of the graphon representing it.  It is shown
in~\cite{MR883646} and~\cite[Theorem D.5]{Janson} that
\begin{equation}
\label{eq:EntGraphGraphon}
\lim_{n \to \infty} \frac{\Ent(G(n,W))}{{n \choose 2}}  = \Ent(W).
\end{equation}
A graphon is called \emph{random-free} if it is $\{0,1\}$-valued almost
everywhere. Note that a graphon $W$ is random-free if and only if $\Ent(W)=0$,
which by (\ref{eq:EntGraphGraphon}) is
equivalent to $\Ent(G(n,W))=o(n^2)$. Our first theorem shows that this is sharp
in the sense that the growth of $\Ent(G(n,W))$ for random-free graphons $W$ can be arbitrarily close to quadratic.

\begin{theorem}
\label{thm:rfreelargeEnt}
Let $\alpha: \mathbb{N} \to \mathbb{R}_{+}$ be a function with $\lim_{n \to \infty} \alpha(n)=0$. Then there exists a random-free
graphon $W$ such that $\Ent(G(n,W))=\Omega(\alpha(n)n^2)$.
\end{theorem}

A graph property $\QQ$ is called \emph{random-free} if every $W \in
\overline{\QQ}$ is
random-free. Our next theorem shows that in contrast to
Theorem~\ref{thm:rfreelargeEnt}, when a graphon $W$ is the limit of a sequence of graphs with a random-free property, then
$\Ent(G(n,W))$ cannot grow faster than $O(n\log n)$.

\begin{theorem}\label{t:EntRandFree}
Let $\QQ$ be a random-free property, and let $W$ be the limit of a sequence of
graphs in $\QQ$. Then
$\Ent(G(n,W)) = O(n \log n)$.
\end{theorem}

\begin{remark}
We defined $G(n,W)$ as a \emph{labeled} graph in $\LL_n$. Both Theorems~\ref{thm:rfreelargeEnt} and~\ref{t:EntRandFree} remain valid if we consider
the random variable $G_u(n,W)$ taking values in $\UU_n$ obtained from $G(n,W)$ by forgetting the labels. Indeed,
$ \Ent(G_u(n,W)) = \Ent(G(n,W)) - \Ent(G(n,W)\: |\: G_u(n,W))$ and $\Ent(G(n,W) \: | \: G_u(n,W)=H) = O(n\log n)$ for every $H \in \UU_n$. It follows that
$$\Ent(G(n,W)) - O(n \log n) \leq \Ent(G_u(n,W)) \leq \Ent(G(n,W)).$$

\end{remark}

\section{Proof of Theorem~\ref{thm:rfreelargeEnt}}
For every positive integer $m$, let $F_{m}$ denote the unique bigraph $([m],[2^m],E)$ with the property that the vertices in $[2^m]$ all have different sets of neighbors.
The \emph{transversal-uniform} graph is the unique graph (up to an isomorphism) with vertex set $\mathbb{N}$ which satisfies the following property.
The vertices are partitioned into sets $\{A_i\}_{i=1}^\infty$ with  $\log |A_i|=\sum_{j=1}^{i-1} |A_{i-1}|$. There are no edges inside $A_i$'s, and for every $i$, the bigraph induced by $(\cup_{j=1}^{i-1} A_j, A_i)$ is isomorphic to $F_{\sum_{j=1}^{i-1} |A_j|}$.

Let $\II=\{I_i\}_{i \in \NN}$ be a partition of $[0,1]$ into intervals. We define its corresponding \emph{transversal-uniform} graphon $W_\II$  by assigning weights $|I_i|/|A_i|$ to all the vertices in $A_i$ in the  transversal-uniform graph $G_{U}$ described above. More precisely, we partition each $I_i$ into $|A_i|$ equal size intervals (corresponded with elements in $A_i$), and mapping all the points in each  of these subintervals to its corresponding vertex in $A_i$. This  measurable surjection $\pi_\II:[0,1] \to \mathbb{N}$, together with the transversal-uniform graph described above defines the transversal-uniform graphon $W_\II$ by setting
$$W_\II(x,y)=\begin{cases}1 & \text{if $\pi(x)\pi(y) \in E(G_U)$,}
\\
0 &\text{if $\pi(x)\pi(y) \not \in E(G_U)$.} \end{cases}$$
Note that by construction $W_\II$ has the following property. Let $s<k$ be positive integers, and $x_1,\ldots,x_s \in \cup_{i<k} I_i$  belong to pairwise distinct intervals in $\II$. For every $f:[s] \to \{0,1\}$, we have
$$
\Pr[ \forall i, \ W_\II(x_i,y)=f(i) \: | \: y \in I_{k} ] = \frac{1}{2^{s}},
$$
where $y$ is a random variable taking values uniformly in $[0,1]$.
It follows that for every graph $H$ on $s$ vertices,
\begin{equation}\label{eq:nextStep2}
\Pr[ G(x_1,\ldots,x_s, W_\II) = H \: | \: \forall i, \ x_i \in I_{k_i} ] = \frac{1}{2^{s \choose 2}},
\end{equation}
where $x_1,\ldots,x_s$ are now i.i.d. random variables taking values uniformly in $[0,1]$, and $k_1,k_2,\ldots, k_s$ are distinct natural numbers.

%
%
We translate (\ref{eq:nextStep2}) into a lower bound on (conditional) entropy of transversal-uniform graphons. First we need a simple lemma.

\begin{lemma}\label{lem:transUniform} Let  $W_{\II}$ be a transversal-uniform graphon, and $\phi:[n] \to [0,1]$ be a uniformly random map.
For every $\rho:[n] \to \mathbb{N}$, we have
$$\Ent(G(\phi(1),\ldots,\phi(n),W_{\II}) \: | \pi_\II \circ \phi = \rho) \geq \binom{|\Image(\rho)|}{2}.$$
\end{lemma}
\begin{proof}
Pick a set  of representatives $K \subseteq [n]$ so that $\rho|_K:K \to \Image(\rho)$ is a bijection. Equation (\ref{eq:nextStep2}) implies that for every graph $H$ with $V(H)=K$,
\begin{equation*}\label{eq:uniformity}
\Pr[G(\phi(1),\ldots,\phi(n),W_{\II})[K]=H \: | \:  \pi_\II \circ \phi = \rho ]=\frac{1}{2^{\binom{|\Image(\rho)|}{2}}}.
\end{equation*}
Therefore,
$$\Ent(G(\phi(1),\ldots,\phi(n),W_{\II})  \: | \: \pi_\II \circ \phi = \rho ) \geq
\Ent(G(\phi(1),\ldots,\phi(n),W_{\II})[K] \: | \: \pi_\II \circ \phi = \rho ) = \binom{|\Image(\rho)|}{2}.$$
\end{proof}

In the proof of Theorem~\ref{thm:rfreelargeEnt} below we will make use of the following well-known  inequality about conditional entropy. For discrete random variables $X$ and $Y$,
\begin{equation}\label{eq:condEntropy}
\Ent(X \: | \: Y) := \sum_{y \in \supp(Y)} \Pr[Y=y]\Ent(X \: | \: Y=y) \leq \Ent(X).
\end{equation}

\begin{proof}[Proof of Theorem~\ref{thm:rfreelargeEnt}.]
For every positive integer $k$, define
$$g_k:=  \max \left\{ \{2^{k+5}\} \cup \{n  \: | \: \alpha(n) > 2^{-2k-9} \right\}\}.$$
The numbers $g_k$ are well-defined, as the condition $\lim_{n \to \infty}\alpha(n)=0$ implies that the set  $\{n \: | \:\alpha(n) > 2^{-2k-9}\}$ is finite. Define the sums $G_k := \sum_{i=1}^k g_k$, and set $\beta_i  = \frac{1}{g_k2^k}$ for all the $g_k$ indices $i \in \big(G_{k-1},G_k\big]$. Let $\II=\{I_i\}_{i \in \NN}$
be a partition of $[0,1]$ into intervals with $|I_i|=\beta_i$, and let $W_{\II}$ be the corresponding transversal-uniform graphon.

Consider a sufficiently large $n \in \NN$, and let $k \in \NN$ be chosen to be maximum so that $2^{k+4} \leq n$ and  $\alpha(n) \leq 2^{-2k-7}$.
We have $n < 2^{k+5}$ or $\alpha(n) > 2^{-2k-9}$. Therefore  $n \leq g_k$ by the definition of $g_k$. Let $\phi:[n] \to [0,1]$ be random and uniform. By Lemma~\ref{lem:transUniform}, for any fixed $\rho:[n]\to \mathbb{N}$, we have
$$\Ent(G(\phi(1),\ldots,\phi(n),W_\II) | \pi_\II \circ \phi = \rho) \ge {|\Image(\rho)| \choose 2}.$$
Thus
\begin{equation}
\label{eq:rhophi}
\Ent(G(n,W_\II))  \ge \Ent(G(n,W_\II) | \pi_\II \circ \phi ) \ge \Pr\left[|\Image(\pi_\II \circ \phi )| \ge n2^{-k-2} \right] {n2^{-k-2} \choose 2}.
\end{equation}
Define the random variable $X:=|\Image(\pi_\II \circ \phi) \cap (G_{k-1},G_k]| \le |\Image(\pi_\II \circ \phi )|$. We have
$$\Ex[X]= \sum_{i \in (G_{k-1},G_k]} \Pr[\phi^{-1}(I_i) \neq \emptyset] =\sum_{i \in (G_{k-1},G_k]} (1-(1-\beta_i)^n)= g_k \left(1-\left(1-\frac{1}{g_k2^k}\right)^n\right)
\ge n2^{-k-1},$$
where we used the fact that $g_k2^k\ge 2n$ and that $(1-x)^n \le 1-nx+n^2x^2 \le 1-nx/2$ for $x \in [0,1/2n]$. As
the events $\phi^{-1}(I_i) \neq \emptyset$ and $\phi^{-1}(I_j) \neq \emptyset$ are negatively correlated for $i \neq j$, we have $\Var[X] \le \Ex[X]$.
Hence by Chebyshev's inequality
\begin{align*}
\Pr\left[|\Image(\pi_\II \circ \phi )| \ge n2^{-k-2} \right] &\ge \Pr\left[X \ge n2^{-k-2} \right] \ge 1-
\Pr\left[|X -\Ex[X]| \ge \frac{\Ex[X]}{2} \right] \\ 
&\ge 1- \frac{4 \Var[X]}{\Ex[X]^2} \ge 1 -\frac{4 }{n2^{-k-2}} \ge \frac{1}{2}.
\end{align*}
Substituting in (\ref{eq:rhophi}) we obtain
$$ \Ent(G(n,W_\II))  \ge  \frac{1}{2} {n2^{-k-2} \choose 2} \ge n^2 2^{-2k-7} \ge \alpha(n) n^2,$$
as desired.
\end{proof}

\section{Proof of Theorem~\ref{t:EntRandFree}}
In~\cite{LovaszSzegedyTop} Lov\'asz and Szegedy obtained a combinatorial
characterization of random-free graph properties. To state this result it is convenient
to distinguish between bipartite graphs and bigraphs. A \emph{bipartite}
graph is a graph $(V, E)$ whose node set has a partition into two classes such that all edges connect
nodes in different classes. A \emph{bigraph} is a triple $(U_1, U_2, E)$ where $U_1$ and $U_2$ are finite sets and
$E \subseteq U_1 \times  U_2$. So a bipartite graph becomes a bigraph if we fix a bipartition and specify which
bipartition class is first and second. On the other hand, if $F = (V, E)$ is a graph, then $(V, V, E')$
is an associated bigraph, where $E' = \{(x, y) : xy \in E\}$.

If $G = (V, E)$ is a graph, then an induced sub-bigraph of $G$ is determined by two (not necessarily disjoin) subsets
$S,T \subseteq V$, and its edge set consists of those pairs $(x, y) \in S \times T$ for which $xy \in E$ (so this is an
induced subgraph of the bigraph associated with $G$).

For a bigraph $H=(U_1,U_2,E)$ and a graphon $W$,
analogous to the definition of the induced density of a graph in a graphon, we define
$$p^{\bsf}(H;W) = \Ex\left[\prod_{\substack{u \in U_1,\: v  \in U_2\\ uv \in E}} W(x_u,y_v) \prod_{\substack{u \in U_1,\: v  \in U_2\\ uv
\in (U_1 \times U_2) \setminus E}} (1-W(x_u,y_v))\right],$$
where $\{x_u\}_{u \in U_1}, \{y_v\}_{v \in U_2}$ are independent random variables taking values in $[0,1]$ uniformly.
Now we are ready to state Lov\'asz and Szegedy's characterization of random-free graph properties.

\begin{theorem}\cite{LovaszSzegedyTop}\label{LS-Hfree}
A graph property $\QQ$ is random-free if and only if there exists a bigraph $H$ such that
$p^{\bsf}(H;W)=0$ for all $W \in \overline{\QQ}$.
\end{theorem}

The following lemma is due to Alon,  Fischer, and Newman (See~\cite[Lemma 1.6]{MR2341924}).
\begin{lemma}\cite{MR2341924}\label{lem:Alon}
Let $k$ be a fixed integer and let $\delta > 0$ be a small real. For every graph $G$,
either there exists stepfunction graphon $W'$ with $r \le \left(\frac{k}{\delta} \right)^{O(k)}$ steps  such
that $\delta_1(W_G,W') \le \delta$, or for every bigraph $H$ on $k$ vertices
$p^b(H;G) \ge \left(\frac{\delta}{k}\right)^{O(k^2)}$.
\end{lemma}
Every random-free graphon $W$ can be approximated arbitrarily well in the $\delta_1$ distance with $W_G$ for some graph $G$, and furthermore,
for every fixed $H$, the function $p^\bsf(H,\cdot)$ is continuous in the $\delta_1$ distance.
Thus Lemma~\ref{lem:Alon} can be generalized to random-free graphons.
\begin{corollary}\label{c:AlonGraphon}
Let $k$ be a fixed integer and let $\delta > 0$ be a small real. For every
random-free graphon $W$, either there exists a stepfunction  graphon $
W'$ with $r \le \left(\frac{k}{\delta} \right)^{O(k)}$ steps such
that $\delta_1(W,W') \le \delta$, or for every bigraph $H$ on $k$ vertices
$p^b(H;G) \ge \left(\frac{\delta}{k}\right)^{O(k^2)}$.
\end{corollary}
Next we will prove two simple lemmas about  entropy.
\begin{lemma}\label{l:ApproxEnt}
Let $\mu_1$ and $\mu_2$ be two discrete probabilistic distributions on a finite set $\Omega$.
Then $$|\Ent(\mu_1) - \Ent(\mu_2)| \leq |\Omega|h\left(\frac{\|\mu_1 - \mu_2\|_1}{|\Omega|}\right).$$
\end{lemma}
\begin{proof}
Define $0 \log 0 := \lim_{x \to 0} x \log x =0$. By taking the derivative with respect to $x$, for fixed
$d$ we see that $(x+d)\log(x+d)-x\log x$ is monotone for $0 \leq x \leq 1-d$. Therefore,
for  $x_1,x_2 \in [0,1]$ we have
$$|x_2 \log x_2 - x_1 \log x_1| \le \max\{-|x_2-x_1| \log |x_2-x_1|,-(1-|x_2-x_1|)\log(1-|x_2-x_1|)\} \le h(|x_2-x_1|).$$
Thus
\begin{eqnarray*}
|\Ent(\mu_1) - \Ent(\mu_2)| &=& \left| \sum_{x \in \Omega} \mu_1(x) \log \mu_1(x)- \mu_2(x) \log \mu_2(x) \right| \\
&\le& \sum_{x \in \Omega} h(|\mu_1(x)-\mu_2(x)|) \le  |\Omega|h\left(\frac{\|\mu_1 - \mu_2\|_1}{|\Omega|}\right),
\end{eqnarray*}
where the last inequality is by concavity of the binary entropy function $h$.
\end{proof}

\begin{lemma}\label{l:ApproxGraphon}
Let $W_1$ and $W_2$ be two graphons, and let $\mu_1$ and $\mu_2$ be the
probability distributions on $\LL_n$ induced by $G(n,W_1)$ and $G(n,W_2)$,
respectively. Then
$$\|\mu_1 - \mu_2\|_1 \le n^2  \delta_1(W_1, W_2).$$
\end{lemma}
\begin{proof}
Let $x_1,\ldots,x_n$ be i.i.d. uniform random variables with values in $[0,1]$. Note
\begin{eqnarray*}
\|\mu_1 - \mu_2\|_1 &\le& \Pr \left[ G(x_1,\ldots,x_n,W_1) \neq G(x_1,\ldots,x_n,W_2) \right] \\
&\le& \Ex \left[\sum_{i \neq j} |W_1(x_i,x_j)-W_2(x_i,x_j)| \right] \le n^2 \|W_1-W_2\|_1.
\end{eqnarray*}

\end{proof}

\begin{proof}[Proof of Theorem~\ref{t:EntRandFree}]
Since $\QQ$ is random-free, by Theorem~\ref{LS-Hfree}, there exists a bigraph $H$ such that
$p^{\bsf}(H;W)=0$ for all $W \in \overline{\QQ}$. Applying Corollary~\ref{c:AlonGraphon} with $\delta=1/n^5$ shows that there exists
a stepfunction graphon $W'$ with $n^{O(1)}$ steps satisfying $\|W-W'\|_1 \le \delta$.
Then since $|\LL_n| \le 2^{n^2}$, Lemmas~\ref{l:ApproxEnt} and~\ref{l:ApproxGraphon} imply
\begin{align*}
| \Ent(G(n,W'))&-\Ent(G(n,W))| \leq  2^{n^2}h\left(\frac{n^2\delta}{2^{n^2}}\right)\\
&=-2^{n^2} \left(\frac{n^2\delta}{2^{n^2}} \log\left(\frac{n^2\delta}{2^{n^2}} \right)+ \left(1-\frac{n^2\delta}{2^{n^2}}\right) \log\left(1-\frac{n^2\delta}{2^{n^2}} \right) \right)\\
&\leq n^4\delta+n^2\delta(-2\log n -\log{\delta}) + 2^{n^2}\cdot 2 \frac{n^2\delta}{2^{n^2}}=o(1).
\end{align*}
Since $W'$ is random-free and it has $n^{O(1)}$ steps, $|\supp(G(n,W'))| = n^{O(n)}$. Consequently
$\Ent(G(n,W'))=O(n \log{n})$.
\end{proof}

\section{Concluding remarks}

\noindent {\bf 1.} Note that if $W$ is a random-free stepfunction, then $\Ent(G(n,W)) =O(n)$.
In~\cite{Janson} it is conjectured that the converse is also true. That is
$\Ent(G(n,W))=O(n)$ if and only if $W$ is equivalent to a random-free   stepfunction. The following simple example
disproves this conjecture.

Let $\mu$ be the probability distribution on $\mathbb{N}$ defined by $\mu(\{i\})=2^{-i}$.
Consider the random variable $X=(X_1,\ldots,X_n) \in \mathbb{N}^n$ where $X_i$ are i.i.d. random variables with distribution $\mu$.
We have $\Ent(X_i)=\sum_{i=1}^\infty 2^{-i} i = 2$. Hence $\Ent(X) \le \sum \Ent(X_i) = 2n$.

Partition $[0,1]$ into intervals $\{I_i\}_{i=1}^\infty$ where $|I_i|=2^{-i}$. Let $W$ be the graphon that is constant $1$ on $\cup_{i=1}^\infty
I_i \times I_i$ and $0$ everywhere else. Note that $$\Ent(G(n,W)) \le \Ent(X) \le 2n.$$
Therefore $G(n,W)$ has linear entropy.

It remains to verify that $W$ is not equivalent to a stepfunction. This follows immediately from the fact that $W$ has infinite rank as a kernel. It can also be verified in a more combinatorial way: A \emph{homogenous set} of vertices in a graph $H$ is a set of vertices which
are either all pairwise adjacent to each other, or all pairwise non-adjacent. If $W$ is equivalent to a step-function with $k$ steps, then every $H \in \supp(G(n,W))$  cleary contains a homogenous set of size at least $n/k$.  On the other hand, if $H \in \LL_{n^2}$ is a disjoint union of $n$ complete graphs on $n$ vertices, then the largest homogenous set in $H$ has size $n$, but $H \in \supp(G(n^2,W))$ by construction.

\vskip10pt
\noindent {\bf 2.}  Theorem~\ref{t:EntRandFree} shows that when $W$ is a limit of a random-free property, then the entropy of $G(n,W)$ is small.
However, the support of $G(n,W)$ can be comparatively large. For every $\epsilon >0$, we construct examples
for which $\log(|\supp(G(n,W))|)=\Omega(n^{2-\epsilon})$.  Note that Theorem~\ref{t:EntRandFree} implies that
$G(n,W)$ is far from being uniform on the support in these examples, as the entropy of a uniform random variable with support of size $2^{\Omega(n^{2-\epsilon})}$  is $\Omega(n^{2-\epsilon})$.

Let us now describe the construction. Let $\QQ$ be the set of graphs that do not contain $K_{t,t}$ as a subgraph.  Partition $[0,1]$ into intervals $\{S_i\}_{i=1}^\infty$ with non-zero lengths, and let $\{H_i\}_{i=1}^\infty$ be an enumeration of graphs in $\QQ$. Define $W$ to be the graphon that is $0$ on $S_i \times S_j$ for $i \neq j$, and is equivalent to $W_{H_i}$ (scaled properly) on $S_i \times S_i$. By construction $p(H;W)>0$ if $H \in \QQ$.
Thus $|\supp(G(n,W))| \geq  |\QQ_{n}|$. Since there exists $K_{t,t}$-free graphs with $n^{2-2/t}$ edges (See e.g.~\cite[p. 316, Thm. VI.2.10]{MR506522}),
we have $|\QQ_n| \ge 2^{n^{2-2/t}}$.

It remains to show that $W$ is a limit of graphs in some random-free property. Unfortunately, $W \not \in \overline{\QQ}$. We construct a larger random-free property $\QQ'$ so that $W \in \overline{\QQ'}$ as follows.

Fix a bigraph $B$, so that the corresponding graph contains $K_{t,t}$ as a subgraph and is connected. Suppose further that no two vertices of $B$ have the same neighborhood. Note that such a bigraph trivially exists. For example, one can take $B=(V_1 \cup U_1,V_2 \cup U_2,E)$ so that $V_1,U_1,V_2,U_2$ are disjoint sets of size $t$, every vertex of $V_1$ is joined to every vertex of $V_2$, and the edges between $V_1$ and $U_2$, as well as the edges between $U_1$ and $V_2$, form a matching of size $t$.  Let $\QQ' \supseteq \QQ$ be the set of graphs not  containing $B$ as an induced sub-bigraph. Then $\QQ'$ is random-free by Theorem~\ref{LS-Hfree}, as $p^{\bsf}(B,W')=0$ for every $W' \in \overline{\QQ'}$.

Let $r=|V(B)|$ and suppose that $G=G(x_1,x_2,\ldots,x_r, W)$ contains $B$ as an induced sub-bigraph. Then there exists $i$ so that
$x_1,x_2,\ldots,x_r \in S_i$, as $G$ is connected. It follows further that $G$ is an induced subgraph of $H_i$, as no two vertices of $G$ have the same neighborhood. Thus $G$ contains no $K_{t,t}$ subgraph, contradicting our assumption that $G$ contains $B$. We conclude that $\supp(G(n,W)) \subseteq \QQ'$ for every positive integer $n$. By~\cite[Lemma 2.6]{MR2274085} the sequence $\{G(n,W)\}_{n=1}^{\infty}$ converges to $W$ with probability one. Thus $W \in \overline{\QQ'}$, as desired.

\bibliographystyle{alpha}
\bibliography{HNentropy}
\end{document}